\documentclass[letterpaper,10 pt, conference]{ieeeconf}

\usepackage{cite}
\usepackage{graphicx}
\usepackage{amssymb}
\usepackage[cmex 10]{amsmath}
\usepackage{array}
\usepackage{mdwmath}
\usepackage{mdwtab}
\usepackage{balance}
\usepackage{multicol}

\newtheorem{theorem}{Theorem}
\newtheorem{definition}{Definition}
\newtheorem{lemma}{Lemma}
\newtheorem{remark}{Remark}

\IEEEoverridecommandlockouts                              
\title{\LARGE \bf
A New Stability Result for  the Feedback Interconnection of Negative
Imaginary  Systems with a Pole at the Origin }

\author{Mohamed A. Mabrok, Abhijit G. Kallapur, Ian R. Petersen, and Alexander Lanzon
\thanks{This work was supported by the Australian Research Council}
\thanks{Mohamed Mabrok, Abhijit Kallapur and Ian Petersen are with the School of Engineering and Information Technology,
        University of New South Wales at the Australian Defence Force Academy
        Canberra ACT 2600, Australia
        {\tt\small abdallamath@gmail.com, abhijit.kallapur@gmail.com, i.r.petersen@gmail.com}}%
\thanks{Alexander Lanzon is with the Control Systems Centre, School of Electrical and Electronic Engineering,
        University of Manchester, Manchester M13 9PL, United Kingdom
        {\tt\small Alexander.Lanzon@manchester.ac.uk}}%
}

\begin{document}
\maketitle

\begin{abstract}
This paper is concerned with  stability conditions for the  positive
feedback interconnection of negative imaginary systems. A
generalization of the negative imaginary lemma is derived, which
remains true even if the transfer function has poles on the
imaginary axis including the origin. A  sufficient condition for the
internal stability of a feedback interconnection for NI systems
including   a pole at the origin is given and an illustrative
example is presented to support the result.
\end{abstract}


\section{Introduction}

Structural modes in machines and robots, ground and aerospace
vehicles, and precision instrumentation, such as atomic force
microscopes and optical systems, can limit the ability of control
systems to achieve the desired performance \cite{petersen2010}. This
problem is simplified to some extent by using force actuators
combined with collocated measurements of velocity, position, or
acceleration.

The use of  force actuators combined with  velocity  measurements
has been studied using the  positive real (PR) theory  for linear
time invariant (LTI) systems; e.g., see
\cite{anderson-bk1973,brogliato-bk2007}. PR systems, in the
single-input single-output (SISO) case, can be defined as systems
where the real part of the transfer function is nonnegative. Many
systems that dissipate energy fall under the category of PR systems.
For instance, they can arise in electric circuits with linear
passive components and magnetic couplings. In spite of its success,
a drawback of the PR theory is the requirement for the relative
degree of the underlying system transfer function to be either zero
or one \cite{brogliato-bk2007}. Hence, the control of flexible
structures with force actuators combined with position measurements,
cannot use the theory of  PR systems.

Lanzon and Petersen introduce a new class of systems in
\cite{lanzon2007,lanzon2008} called negative imaginary (NI) systems,
which has fewer restrictions on the relative degree of the system
transfer function than in the PR case. In the SISO case, such
systems are defined by considering the properties of the imaginary
part of the  transfer function $G(j\omega) = D + C (j\omega I -
A)^{-1} B,$ and requiring  the condition $j\left( G(j\omega
)-G(j\omega )^{\ast }\right) \geq 0$ for all $\omega\in(0,\infty)$.

In general, NI systems are stable systems having a phase lag between
0 and $-\pi$ for all $\omega
> 0$. That is, their Nyquist plot lies below  the real axis when the frequency
varies in the open interval  $(0,\infty)$ (for strictly
negative-imaginary systems, the Nyquist plot should not touch the
real axis except at zero frequency or at infinity). This is similar
to PR systems where the  Nyquist plot is constrained to lie in the
right half of the complex plane
\cite{anderson-bk1973,brogliato-bk2007}. However, in contrast to PR
systems,  transfer functions for NI systems can have relative degree
more than unity.

NI systems can be transformed into  PR systems and vice versa under
some technical assumptions. However, this equivalence is not
complete. For instance, such a transformation applied to  a strictly
negative imaginary (SNI) system  always leads to a non-strict PR
system. Hence, the passivity  theorem
\cite{anderson-bk1973,brogliato-bk2007} cannot capture the stability
of the closed-loop  interconnection of an NI and an SNI system. In
addition, any controller design approach based on strictly PR
synthesis cannot be used for the control of an NI system
irrespective of whether it is strict or
non-strict. 
Also, transformations of NI systems to bounded-real systems for
application of the small-gain theorem  suffers from the exact same
difficulty of giving a non-strict bounded real system despite the
original system being SNI; see \cite{song2010} for details.


Many practical systems  can be consider as  NI systems. For example,
when considering the transfer function from a force actuator to a
corresponding  collocated position sensor (for instance,
piezoelectric sensor) in a lightly damped structure
\cite{lanzon2007,lanzon2008,fanson1990,petersen2010,Bhikkaji2009,Yong2010}.
Also, stability results for interconnecting systems with an NI
frequency response have been applied to decentralized control of
large vehicle platoons in \cite{hagen2010}. Here, the authors
discuss the availability of various  designs to enhance the robust
stability of the system with respect to small variations in
neighbor-coupling gains.

 NI systems theory has been extended by Xiong et. al. in
\cite{xiong2009,xiong2009a,xiong2010}  by allowing for simple poles
on the imaginary axis of the complex plane except at the origin.
Furthermore, NI controller synthesis has also been discussed in
\cite{lanzon2007,lanzon2008}. In addition, it has been shown in
\cite{lanzon2007,lanzon2008} that a necessary and sufficient
condition for the internal stability of a positive-feedback
interconnection of an NI system with transfer function matrix $M(s)$
and  an SNI system with transfer function matrix $N(s)$ is given by
the DC gain condition $\lambda_{max}(M(0)N(0))<1$. Here, the
notation $\lambda_{max}(\cdot)$ denotes the maximum eigenvalue of a
matrix with only real eigenvalues.

A generalization of  the  NI lemma in \cite{xiong2009a,xiong2010} to
include a simple pole at the origin was presented in
\cite{mabrok2011}. In \cite{mabrok2011}, stability analysis for a
spacial class of generalized NI systems with the inclusion of an
integrator connected in parallel  with an NI system   was discussed.
The assumption in \cite{mabrok2011} restricts the application of the
proposed stability result to NI systems which can be decomposed into
the parallel connection of
 an NI system and an
integrator.

In this paper, we extend the results in
\cite{lanzon2007,lanzon2008,xiong2009,xiong2009a,xiong2010,petersen2010,mabrok2011}
  for NI systems to allow for the existence of a pole at the origin with a more general structure than allowed in the result of  \cite{mabrok2011}. This extension allows us to stabilize
   any NI system with a pole at the origin
  without any
  parallel decomposition  assumption. Also, stabilizing NI systems with a pole at the origin can be used for
  controller design  with  integral action.


This paper is further organized as follows: Section
\ref{sec:preliminaries} introduces the concept of PR and NI systems
and presents a relationship between them. The main results of this
paper are presented in Section  \ref{sec:main-results}. Section
\ref{sec:numerical-example} provides a numerical example and the
paper is concluded with  a summary and remarks on future work in
Section \ref{sec:conclusion}.

\section{Preliminaries}
\label{sec:preliminaries}

In this section, we introduce the definitions  of PR and NI systems.
We also present a lemma describing the transformation between PR and
NI systems, and some technical results which will be used in
deriving the main results of the paper.

 The definition of PR systems has been motivated
by the study of linear electric circuits composed of resistors,
capacitors, and inductors. For a detailed discussion of PR systems,
see \cite{anderson-bk1973,brogliato-bk2007} and references therein.

\begin{definition}
A square transfer function matrix $F(s)$ is positive real if:
\begin{enumerate}
    \item $F(s)$ has no pole in $Re[s]>0$.
    \item $F(j \omega)+ F(j \omega)^{\ast }\geq 0$ for all positive real $j \omega$ such that $j \omega $ is not a pole of $F(j \omega)$.
   \item If $j \omega_{0} $, finite or infinite, is a pole of $F(j \omega)$, it is a simple pole and the
corresponding residual matrix $K_{0}=\underset{%
s\longrightarrow j\omega _{0}}{\lim }(s-j\omega _{0})F(s)$ is
positive semidefinite Hermitian.
\end{enumerate}
\end{definition}

To establish the main results of this paper,  we consider  a
generalized definition for NI systems which allows for a simple pole
at the origin as follows:
\begin{definition}\label{Def:NI}
A square transfer function matrix $G(s)$ is NI if the following
conditions are satisfied:
\begin{enumerate}
\item $G(s)$ has no pole in $Re[s]>0$.
\item For all $\omega \geq0$ such that $j\omega$ is not a pole of $G(s )$, $j\left( G(j\omega )-G(j\omega )^{\ast }\right) \geq 0$.
\item if $s=j\omega _{0}$ is a pole of $G(s)$ then it is a simple pole. Furthermore  if $\omega _{0}>0$, the residual matrix $K_{0}=\underset{%
s\longrightarrow j\omega _{0}}{\lim }(s-j\omega _{0})jG(s)$ is
positive semidefinite Hermitian.
\end{enumerate}
\end{definition}

\begin{definition}
A square transfer function matrix $G(s)$ is SNI if  the following
conditions are satisfied:
\begin{enumerate}
\item $G(s)$ has no pole in $Re[s]\geq0$.
\item For all $\omega >0$, $j\left( G(j\omega )-G(j\omega )^{\ast }\right) > 0$.
\end{enumerate}
\end{definition}

Due to  advances in the theory of  PR systems and the complementary
definitions of PR and NI systems, it is useful to establish a lemma
which considers  the relationship between these notions  to further
develop the theory of  NI systems. In order to do so,
 we consider the possibility of having a simple pole at the origin, and relax the condition $\det (A)\neq0$  considered  in
\cite{lanzon2008,xiong2009,xiong21010jor}. This leads to a
modification of  the relationship between PR and NI systems as
follows:

\begin{lemma}(see also \cite{mabrok2011}) \label{PR-NI-lemma}
Given a real rational  proper transfer function matrix $G(s)$ with
state space realization $
\begin{bmatrix}
\begin{array}{c|c}
A & B \\ \hline C & D
\end{array}
\end{bmatrix}
$  and the transfer function matrix $\tilde{G}(s)=G(s)-D $, the
transfer function matrix $G(s)$ is NI if and only if the transfer
function matrix $F(s)=s\tilde{G}(s)$ is PR. Here, we assume that any
pole zero cancellation which occurs in $s\tilde{G}(s)$ has been
carried out to obtain $F(s)$.
\end{lemma}

\begin{proof}
(Necessity) It is straightforward  to show that if $\tilde{G}(s)$ is
NI then $G(s)$ is NI and vice-versa. Suppose that $j \left(
\tilde{G}(j\omega )-\tilde{G}(j\omega )^{\ast }\right) \geq 0, $ for
all $\omega
>0$ such that $j\omega $  is not a pole of $G(s)$. Then
 given any such $\omega > 0$,
$ F(j\omega )+F(j\omega )^{\ast } =j\omega \left( \tilde{G}(j\omega
)-\tilde{G}(j\omega )^{\ast }\right) \geq 0, $ and $\overline{\left(
F(j\omega )+F(j\omega )^{\ast }\right) }\geq 0 $.  This means that $
F(-j\omega )+F(-j\omega )^{\ast } \geq 0$ for all  $\omega >0$ which
implies that $ F(j\omega )+F(j\omega )^{\ast } \geq 0$ for all
$\omega <0$ such that $j\omega$ is not a pole of $G(s)$. Hence,
$\left( F(j\omega )+F(j\omega )^{\ast }\right) $ $\geq 0$ for all
$\omega \in (-\infty ,\infty )$ such that $j\omega $ is not a pole
of $\tilde{G}(j\omega)$.

Now, consider the case where $j\omega_{0}$ is a pole of
$\tilde{G}(s)$ and $\omega _{0}=0.$ Since $\tilde{G}(s)$ has only a
simple pole at the origin, $F(s)=s\tilde{G}(s)$ will have no pole at
the origin because of the pole zero cancellation. This implies that
$F(0)$ is finite.
Since $ F(j\omega )+F(j\omega )^{\ast } \geq 0$ for all $%
\omega >0$ and $F(j\omega )$ is continuous, this implies that
$F(0)+F(0)^{\ast }\geq 0$. Also, if $j\omega _{0}$
is a pole of $\tilde{G}(s)$ and $\omega _{0}>0$, then $\tilde{G}(s)$ can be factored as $\frac{1}{s^{2}+\omega _{0}^{2}%
}R(s)$, which according to the definition for NI systems implies
that the residual matrix $ K_{0}=\frac{1}{2\omega _{0}}R(j\omega
_{0})$ is positive semidefinite Hermitian. This implies that
$R(j\omega _{0})=R(j\omega _{0})^{\ast }\geq 0$. Now, the residual
matrix of $F(s)$ at $j\omega _{0}$  with \ $\omega _{0}>0$ is given
by,
\begin{eqnarray*}
\underset{s\longrightarrow j\omega _{0}}{\lim }(s-j\omega _{0})F(s) &=&%
\underset{s\longrightarrow j\omega _{0}}{\lim }(s-j\omega _{0})s
\tilde{G}(s), \\
&=&\underset{s\longrightarrow j\omega _{0}}{\lim }(s-j\omega _{0})s\frac{1}{%
s^{2}+\omega _{0}^{2}}R(s),\\
 &=&\frac{1}{2}R(j\omega _{0})
\end{eqnarray*}%
which is positive semidefinite Hermitian. 
Hence, $F(s)$ is positive real.

 (Sufficiency) Suppose that $F(s)$ is positive real. Then, $F(j\omega )+F(j\omega )^{\ast } \geq 0$ for all $\omega \in
(-\infty ,\infty )$ such that $j\omega$ is not a pole of $F(s)$.
This implies $ j\omega \left( \tilde{G}(j\omega )-\tilde{G}(j\omega
)^{\ast }\right) \geq 0$ for all $\omega \geq0$ such that $j\omega$
is not a pole of $G(s)$. Then
  $\tilde{G}(j\omega )-\tilde{G}(j\omega )^{\ast } \geq 0$
for all such $\omega \in \lbrack 0,\infty )$.  In addition, if $j
\omega_0$ is a pole of $F(s)$, then it follows from the definition
of
PR systems that the residual matrix $\underset{s\longrightarrow j\omega _{0}}{%
\lim }(s-j\omega _{0})F(s)$ is positive semidefinite Hermitian.
Also,
\begin{eqnarray*}
\underset{s\longrightarrow j\omega _{0}}{\lim }(s-j\omega _{0})F(s) &=&%
\underset{s\longrightarrow j\omega _{0}}{\lim }(s-j\omega _{0})s
\tilde{G}(s),  \\
 &=&\omega _{0}\underset{s\longrightarrow j\omega _{0}}{\lim
}(s-j\omega _{0})j \tilde{G}(s).
\end{eqnarray*}%
Then using Definition \ref{Def:NI}, we can conclude that
$\tilde{G}(s)$ is NI.
\end{proof}

\begin{remark}  Note that a pole zero
cancellation at the origin in $F(s) = s \tilde G(s)$ will not affect
the use of the PR lemma when applied to $F(s)$ since the minimality
condition is relaxed in the generalized version of the PR lemma
\cite{lanson2011,scherer1994}.
\end{remark}
%

Now, we present  a generalized NI lemma, which allows for a  pole at
the origin.

 Consider the following LTI system,
\begin{align}
\label{eq:xdot}
&\dot{x}(t) = A x(t)+B u(t), \\
\label{eq:y} &y(t) = C x(t)+D u(t),
\end{align}%
where, $A \in \mathbb{R}^{n \times n},B \in \mathbb{R}^{n \times
m},C \in \mathbb{R}^{m \times n},$ and $D \in \mathbb{R}^{m \times
m}$.

 \begin{lemma}(see also \cite{mabrok2011})\label{NI-lemma}
 Let $
\begin{bmatrix}
\begin{array}{c|c}
A & B \\ \hline C & D
\end{array}
\end{bmatrix}$ be a minimal realization of the transfer function matrix  $%
G(s)\in R^{m\times m}$ for the system in
(\ref{eq:xdot})-(\ref{eq:y}). Then, $G(s)$ is NI if and only if
there exist matrices $P=P^{T}\geq 0$,
 $W\in \mathbb{R}^{m \times m}$, and $L\in \mathbb{R}^{m \times n}$ such that the following LMI is
satisfied:
\begin{align}\label{LMI:PR}
\begin{bmatrix}
PA+A^{T}P & PB-A^{T}C^{T} \\
B^{T}P-CA & -(CB+B^{T}C^{T})%
\end{bmatrix}%
&=%
\begin{bmatrix}
-L^{T}L & -L^{T}W \\ \nonumber
-W^{T}L & -W^{T}W%
\end{bmatrix}\\
&\leq 0.
\end{align}

\end{lemma}
\begin{proof}
 Suppose that $G(s)$ is
NI, which implies from Lemma \ref{PR-NI-lemma} that $F(s)=s
\tilde{G}(s)$ with state space realization $
\begin{bmatrix}
\begin{array}{c|c}
A & B \\ \hline CA & CB
\end{array}
\end{bmatrix}$ is PR. It follows from Corollary 2 and Corollary 3 in
\cite{scherer1994} that there exists a matrix $P=P^{T}\geq 0$, such
that the  LMI in (\ref{LMI:PR}) is satisfied.

On the other hand, suppose that LMI in (\ref{LMI:PR}) is satisfied,
then $F(s)$ is PR via Corollary 1  and Corollary 3  in
\cite{scherer1994}, which implies from Lemma \ref{PR-NI-lemma} that
$G(s)$ is NI.
\end{proof}

In studying the internal stability  of an interconnection of NI and
SNI systems, we shall use the
following SNI lemma:
\begin{lemma}\cite{lanzon2008,xiong2009,xiong21010jor}\label{SNI-lemma}
Suppose that the  proper transfer function matrix $G(s)=C(sI-
A)^{-1}B+D$  with a  minimal realization $
\begin{bmatrix}
\begin{array}{c|c}
A & B \\ \hline C & D
\end{array}
\end{bmatrix}$ is SNI, then
the following conditions are satisfied:
\begin{enumerate}
\item $\det(A)\neq0$, $D=D^{T}$.
\item There exists a square  matrix $P=P^{T}>0$,
 $W\in \mathbb{R}^{m \times m}$ and $L\in \mathbb{R}^{m \times n}$
such that the following LMI is satisfied:
\end{enumerate}
\begin{equation}\label{LMI:SNI-PR}
\begin{bmatrix}
PA+A^{T}P & PB-A^{T}C^{T} \\
B^{T}P-CA & -(CB+B^{T}C^{T})%
\end{bmatrix}%
=%
\begin{bmatrix}
-L^{T}L & -L^{T}W \\
-W^{T}L & -W^{T}W%
\end{bmatrix}.%
\end{equation}
\end{lemma}

Also, consider  the following lemma, which will be used to derive
the main results of this paper in Section \ref{sec:main-results},
\begin{lemma}\cite{lanzon2008} \label{Mtr:Lem}
Given $A\in \mathbb{C}^{n \times n}$ with $j(A-A^{\ast })\geq 0$ and
$B\in \mathbb{C}^{n \times n}$ with $j(B-B^{\ast })>0,$ then $\det
(I-AB)\neq 0$.
\end{lemma}

\section{Main results}
\label{sec:main-results} The key result of this paper is a
generalization of the result in \cite{mabrok2011}, which gives
stability conditions for an interconnection between an NI system
(which may contain  a simple pole at the origin) and an SNI system.
The generalization is stated in Theorem \ref{min:result}. Now,
suppose the transfer function matrix $G_{1}(s)$ with a  minimal
realization $
\begin{bmatrix}
\begin{array}{c|c}
A_{1} & B_{1} \\ \hline C_{1} & D_{1}
\end{array}
\end{bmatrix}$
is  NI, and $ G_{2}(s)$ with a  minimal realization $
\begin{bmatrix}
\begin{array}{c|c}
A_{2} & B_{2} \\ \hline C_{2} & D_{2}
\end{array}
\end{bmatrix}$ is SNI. According to Lemma
\ref{NI-lemma} and Lemma \ref{SNI-lemma}, we have,
\begin{align}
P_{1}A_{1}+A_{1}^{T}P_{1} = -L_{1}^{T}L_{1}, \quad &P_{2}A_{2}+A_{2}^{T}P_{2} =-L_{2}^{T}L_{2},  \nonumber \\
\;P_{1}B_{1}-A_{1}^{T}C_{1}^{T} = -L_{1}^{T}W_{1}, \quad &P_{2}B_{2}-A_{2}^{T}C_{2}^{T}=-L_{2}^{T}W_{2}, \nonumber \\
\label{eq:l2-l3} C_{1}B_{1}+B_{1}^{T}C_{1}^{T} = W_{1}^{T}W_{1},
\quad &C_{2}B_{2}+B_{2}^{T}C_{2}^{T} =W_{2}^{T}W_{2},
\end{align}where $P_{1}\geq0$ and $P_{2}>0$.
 The
internal stability  of the closed-loop  positive-feedback
interconnection of $G_{1}(s)$ and $G_{2}(s)$ can be guaranteed  by
considering the stability of the transfer function matrix,
\begin{equation*}
(I-G_{1}(s)G_{2}(s))^{-1}=\breve{D}+\breve{C}(s
I-\breve{A})^{-1}\breve{B},
\end{equation*}%
where,
\begin{align}
\breve{A} &=
\begin{bmatrix}
A_{1} & B_{1}C_{2} \\
0 & A_{2}%
\end{bmatrix}%
+%
\begin{bmatrix}
B_{1}D_{2} \\
B_{2}%
\end{bmatrix}%
(I-D_{1}D_{2})^{-1}%
\begin{bmatrix}
C_{1} & D_{1}C_{2}%
\end{bmatrix}
\nonumber
\label{mat:A}\\
\breve{B} &=%
\begin{bmatrix}
B_{1}D_{2} \\
B_{2}%
\end{bmatrix}%
(I-D_{1}D_{2})^{-1}, \nonumber \\
\breve{C} &=(I-D_{1}D_{2})^{-1}%
\begin{bmatrix}
C_{1} & D_{1}C_{2}%
\end{bmatrix},\nonumber
\\
\breve{D} &=(I-D_{1}D_{2})^{-1}.
\end{align}%

 Now, consider the following
result, which is the main result of this paper:
\begin{theorem}\label{min:result}
Suppose that $G_{1}(s)$ is strictly proper and  NI and $G_{2}(s)$ is
SNI. Then the closed-loop  positive feedback interconnection between
$G_{1}(s)$ and $G_{2}(s)$ is internally stable if $G_{2}(0)<0$ and
the matrix $A_{1}+B_{1}G_{2}(0)C_{1}$ is not singular.
\end{theorem}
\begin{proof}
 To prove this theorem, we  prove that the
matrix $\breve{A} $ in (\ref{mat:A}) is Hurwitz; i.e., all of its
poles lie in the left-half of the complex plane.

Let $T=%
\begin{bmatrix}
P_{1}-C_{1}^{T}D_{2}C_{1} & -C_{1}^{T}C_{2} \\
-C_{2}^{T}C_{1} & P_{2}%
\end{bmatrix}%
$  be a candidate Lyapunov matrix. Since  $G_{2}(0)<0$,
$P_{1}\geq0$, we claim that
\begin{equation}
\label{inq22}
P_{1}-C_{1}^{T}G_{2}(0)C_{1} >0.
\end{equation}
In order to prove this claim, consider
$M=P_{1}-C_{1}^{T}G_{2}(0)C_{1}\geq0$ and $\mathcal{N}(M)=\{x:Mx=0\}$, where $\mathcal{N}(\cdot)$ denotes the null space. Also, given any $x\in \mathcal{N}$ we have $P_{1}x=0$ and $C_{1}x=0$. Now, consider the equations
\begin{align}
\label{eq:l2l3-1}
    P_{1}A_{1}+A_{1}^{T}P_{1} = -L_{1}^{T}L_{1}, \\
\label{eq:l2l3-2}
    B_{1}^{T}P_{1}-C_{1}A_{1} = -W_{1}^{T}L_{1}
\end{align}
outlined in \eqref{eq:l2-l3}. Now pre-multiplying and post-multiplying \eqref{eq:l2l3-1} by $x^T$ and $x$ respectively, we get,
\begin{equation}
\label{eq:s1}
    L_1 x = 0.
\end{equation}
Also, post-multiplying \eqref{eq:l2l3-1} by $x$ results in
\begin{equation}
\label{eq:s2}
    P_1 A_1 x = 0.
\end{equation}
Subsequently, post-multiplying \eqref{eq:l2l3-2} by $x$, gives
\begin{equation}
\label{eq:s3}
    C_1 A_1 x = 0.
\end{equation}
Now, let $y = A_1 x$, which from \eqref{eq:s2} and \eqref{eq:s3} gives
\begin{equation}
    P_1 y = 0, \qquad C_1 y = 0
\end{equation}
which implies $y \in \mathcal{N}(M)$. Thus, we have established that
\begin{equation}
    A_1 \ \mathcal{N}(M) \subset \mathcal{N}(M) \
    \text{and}\
    \mathcal{N}(M) \subset \mathcal{N}(C_1)
\end{equation}
which leads to the fact that $\mathcal{N}(M)$ is a subset of the
unobservable subspace of $(A_1, C_1)$; e.g., see Chapter 18 of
\cite{rugh-bk1996}. It now follows from the minimality of $(A_1,
B_1, C_1, D_1)$ that  $\mathcal{N}(M) = \{0\}$. Hence, $M = P_1 -
C_1^T G_2(0) C_1 > 0$. This completes the proof of the claim.

Now, using this claim, we have
\begin{align*}
&P_{2}>0\text{ and} \\
&P_{1}-C_{1}^{T}(D_{2}+G_{2}(0)-D_{2})C_{1}
>0,
\\
\\
\Rightarrow&P_{2} >0\text{ and} \\
&P_{1}-C_{1}^{T}D_{2}C_{1}-C_{1}^{T}C_{2}P_{2}^{-1}C_{2}^{T}C_{1}
>0,
\\
\\
\Rightarrow&\begin{bmatrix}
P_{1}-C_{1}^{T}D_{2}C_{1} & -C_{1}^{T}C_{2} \\
-C_{2}^{T}C_{1} & P_{2}%
\end{bmatrix}>0.
\\
\end{align*}
That is, $T > 0$.

 Now, the corresponding Lyapunov inequality is given by,
\begin{align*}
T\breve{A}+\breve{A}^{T}T =& \ \begin{bmatrix}
P_{1}-C_{1}^{T}D_{2}C_{1} & -C_{1}^{T}C_{2} \\
-C_{2}^{T}C_{1} & P_{2}%
\end{bmatrix} \\
&\: \:\ \  \times \begin{bmatrix}
A_{1}+B_{1}D_{2}C_{1} & B_{1}C_{2} \\
B_{2}C_{1} & A_{2}%
\end{bmatrix}
\\
&\: \: +
\begin{bmatrix}
A_{1}+B_{1}D_{2}C_{1} & B_{1}C_{2} \\
B_{2}C_{1} & A_{2}%
\end{bmatrix}%
^{T}\\
& \: \: \ \ \times\begin{bmatrix}
P_{1}-C_{1}^{T}D_{2}C_{1} & -C_{1}^{T}C_{2} \\
-C_{2}^{T}C_{1} & P_{2}%
\end{bmatrix},
\\
=&- \begin{bmatrix}
\left( C_{1}^{T}D_{2}W_{1}^{T}+L_{1}^{T}\right)  & C_{1}^{T}W_{2}^{T} \\
C_{2}^{T}W_{1}^{T} & \left( L_{2}^{T}\right)
\end{bmatrix}\\
  &\: \: \: \times \begin{bmatrix}
\left( W_{1}D_{2}C_{1}+L_{1}\right)  & W_{1}C_{2} \\
W_{2}C_{1} & \left( L_{2}\right)
\end{bmatrix}%
\\
\leq&0.
\end{align*}

This implies  that $\breve{A}$ has all its poles in the closed left
half of the complex plane.
 We now show that  $\det (\breve{A})\neq 0$. Indeed, using the assumption
$(A_{1}+B_{1}G_{2}(0)C_{1})$, we obtain
%
\begin{align*}
\det &(\breve{A})\\ &=\det (A_{2})\det
((A_{1}+B_{1}D_{2}C_{1}-B_{1}C_{2}\left(
A_{2}\right) ^{-1}B_{2}C_{1}) \\
 &=\det (A_{2})\det (A_{1}+B_{1}G_{2}(0)C_{1}) \\
 &=\det (A_{2})\det (A_{1}+B_{1}G_{2}(0)C_{1})\\
&\neq0
\end{align*}since $(A_{1}+B_{1}G_{2}(0)C_{1})$ is non singular and $\det
(A_{2})\neq0$.
Also, using Lemma \ref{Mtr:Lem} and the fact that $G_{1}(s)$ is NI
and $G_{2}(s)$ is SNI, we conclude that
$\det(I-G_{1}(j\omega)G_{2}(j \omega))\neq 0$. This   implies that
$\breve{A}$ has no eigenvalues on the imaginary axis for $\omega
>0$. Hence, the matrix $\breve{A}$ is Hurwitz. This completes the proof of the
theorem.
\end{proof}

\section{Illustrative  Example}
\label{sec:numerical-example} To illustrate the main result of this
paper, consider the SNI transfer function
$G_{2}(s)=\frac{1}{s+3}-1$, which satisfies
$G_{2}(0)=-\frac{2}{3}<0$ and the strictly proper NI transfer
function $G_{1}(s)=\frac{1}{s(s+1)}$,  which has a pole at the
origin. Thus, the assumptions in Theorem \ref{min:result} are
satisfied and we can conclude that the closed-loop system is stable.
Also,  the poles of the closed-loop transfer function corresponding
to $G_{2}(s)$ and $G_{1}(s)$ are the roots of the polynomial
$(1-G_{1}(s)G_{2}(s))=s^{3}+4s^{2}+4s+2$ which are $\{-2.84,
-0.58\pm0.61 i\}$. This verifies that the closed-loop transfer
function is indeed  asymptotically stable.

\section{Conclusion}
\label{sec:conclusion} In this paper,  stability results for a
positive-feedback interconnection of NI systems have been derived. A
generalization of the NI lemma, allowing for  a simple pole at the
origin,  has been used in deriving these results. This work can be
used in  the  controller design  to allow for a broader class of NI
systems than considered previously. Also, the stability result for
an NI system with a pole at the origin connected with an SNI system
can be used for
  controller design including    integral action.
The validity of the main results in this paper  have been
illustrated via a numerical example.

\bibliographystyle{IEEEtran}
\bibliography{cdc2011}
\balance
\end{document}